\documentclass[a4paper,reqno]{amsart}

\usepackage{amssymb}
\usepackage{amsthm}
\usepackage{amsmath}

\usepackage{tikz}

\usepackage{dsfont}

\newcommand{\N}{\mathds{N}}

\newcommand{\R}{\mathds{R}}
\newcommand{\C}{\mathds{C}}

\newcommand{\1}{\mathds{1}}

\newcommand{\pp}{\mathrm{pp}}
\newcommand{\ac}{\mathrm{ac}}

\DeclareMathOperator{\TextRe}{Re}
\DeclareMathOperator{\TextIm}{Im}
\renewcommand{\Re}{\TextRe}
\renewcommand{\Im}{\TextIm}

\newcommand{\M}{\mathcal{M}}

\newcommand{\eps}{\varepsilon}
\DeclareMathOperator{\dom}{dom}

\providecommand{\abs}[1]{\left\lvert#1\right\rvert}
\providecommand{\norm}[1]{\left\lVert#1\right\rVert}
\providecommand{\set}[1]{\left\{ #1\right\}}

\newcommand{\from}{\colon}

\newcommand\llim{
\mathchoice{\vcenter{\hbox{${\scriptstyle{-}}$}}}
{\vcenter{\hbox{$\scriptstyle{-}$}}}
{\vcenter{\hbox{$\scriptscriptstyle{-}$}}}
{\vcenter{\hbox{$\scriptscriptstyle{-}$}}}}
\newcommand\rlim{
\mathchoice{\vcenter{\hbox{${\scriptstyle{+}}$}}}
{\vcenter{\hbox{$\scriptstyle{+}$}}}
{\vcenter{\hbox{$\scriptscriptstyle{+}$}}}
{\vcenter{\hbox{$\scriptscriptstyle{+}$}}}}

\makeatletter

\theoremstyle{plain} 
\newtheorem{theorem}{Theorem}[section]

\newtheorem{proposition}[theorem]{Proposition}
\newtheorem{hypothesis}[theorem]{Hypothesis}
\theoremstyle{definition}
\newtheorem{example}[theorem]{Example}
\newtheorem{definition}[theorem]{Definition}
\newtheorem{remark}[theorem]{Remark}

\makeatother

\usepackage{enumitem}

\numberwithin{equation}{section}

\begin{document}

\title[ac spectrum for Laplacians on radial metric trees and periodicity]{Absolutely continuous spectrum for Laplacians on radial metric trees and periodicity}

\author[J.~Rohleder]{Jonathan Rohleder}

\author[C.~Seifert]{Christian Seifert}

\address{Stockholms universitet \\ Matematik \\
106 91 Stockholm \\
Sweden}
\email{jonathan.rohleder@math.su.se}

\address{TU Hamburg \\ Institut f\"ur Mathematik \\
Am Schwarzenberg-Campus~3 \\
Geb\"aude E \\
21073 Hamburg \\
Germany}
\email{christian.seifert@tuhh.de}

\subjclass[2010]{Primary 34L05, Secondary 34L40, 35Q40}

\keywords{Schr\"odinger operator, quantum graph, tree, absolutely continuous spectrum.}

\begin{abstract}
On an infinite, radial metric tree graph we consider the corresponding Laplacian equipped with self-adjoint vertex conditions from a large class including $\delta$- and weighted $\delta'$-couplings.
Assuming the numbers of different edge lengths, branching numbers and different coupling conditions to be finite, we prove that the presence of absolutely continuous spectrum implies that the sequence of geometric data of the tree as well as the coupling conditions are eventually periodic. 
On the other hand, we provide examples of self-adjoint, non-periodic couplings which admit absolutely continuous spectrum.
\end{abstract}

\maketitle

\section{Introduction}

Differential operators on metric graphs have been studied extensively during the last decades, see, for instance, the surveys~\cite{Berkolaiko2016, BerkolaikoKuchment2013, Kuchment2007, Mugnolo2014} and the references therein. Such operators act as ordinary differential operators on the edges, but exhibit a behavior which is highly nontrivial and substantially different from the one-dimensional case due to the influence of the coupling conditions at the vertices.

The question of quantum mechanical transport on metric graphs is related to the presence of absolutely continuous spectrum of corresponding graph Hamiltonians and has attracted a lot of attention recently, mainly for the Laplacian and for Schr\"odinger operators; 
cf.~\cite{BreuerFrank2009, EvansSolomyak2005, ExnerLipovsky2010,ExnerSeifertStollmann2014, Pankrashkin2009}. 
In this note we study the absolutely continuous spectrum of the Laplacian on a radial metric tree, i.e., a tree whose edge lengths and branchings are radially symmetric with respect to a fixed root vertex. 
These particular graphs exhibit a one-dimenional nature according to the tree structure, but its geometric growth properties are closer to high-dimensional spaces. 
Laplacians on radial trees can be studied by decomposing them into direct sums of ordinary differential operators on half-axes subject to generalized point interactions on an infinite sequence of points, see~\cite{Carlson2000, NaimarkSolomyak2000, SobolevSolomyak2002, Solomyak2004} for the case of standard (also called natural or 
Kirchhoff) vertex conditions and~\cite{ExnerLipovsky2010} for more general self-adjoint vertex couplings.

The aim of this note is to investigate, for a large class of self-adjoint coupling conditions at the vertices, the relation between the presence of absolutely continuous spectrum on the one hand and the periodicity of the tree and the coupling coefficients on the other hand. 
For one-dimensional Schr\"odinger operators it is well-established that absolutely continuous spectrum and ``finitely many configurations'' imply periodicity, see, e.g.,~\cite{CarmonaLacroix1990, KlassertLenzStollmann2011,LenzSeifertStollmann2014}. 
For the Laplacian with standard vertex conditions on a radial tree it was shown in~\cite{ExnerSeifertStollmann2014} that the analogous result is true, i.e., in the presence of finitely many different edge lengths and branching numbers a nonempty absolutely continuous spectrum implies periodicity of the tree, i.e., of the edge lengths and branching numbers. 
In the present note we establish the corresponding result for a wide class of self-adjoint vertex conditions, which includes standard, $\delta$ and weighted $\delta'$ couplings as 
well as their combinations. Our main result, Theorem~\ref{thm:main}, states that, provided one has finitely many different edge lenghts, branching numbers and coupling coefficients, a nonempty absolutely continuous spectrum for the Laplacian on a radial tree may prevail only if the tree and the coupling coefficients are eventually periodic. However, certain extra conditions on the coupling coefficients are imposed since there exist non-periodic self-adjoint vertex couplings such that the corresponding Laplacian has a nonempty absolutely continuous spectrum; see Section~\ref{sec:examples} below. 

By contraposition, the main result of this note yields a method for showing absence of absolutely continuous spectrum, namely ensuring aperiodicity and finitely many configurations; 
such kinds of reasonings are well-known for Schr\"odinger operators modeling aperiodically ordered media.
This result complements the recent works~\cite{BreuerFrank2009,ExnerLipovsky2010}, where absence of absolutely continuous spectrum for so-called sparse trees (where the set of edge lengths is unbounded) was shown.

The outline of this note is as follows. In Section~\ref{sec:prelim_main} we introduce the model, i.e.\ the Laplacians on radial metric tree graphs, describe the set of coupling conditions we are dealing with, and state our main result. The two subsequent sections contain material preliminary to the proof of the main result. Indeed, in Section \ref{sec:unitary_equivalence} we recall the direct sum decomposition into half-line operators from~\cite{ExnerLipovsky2010}. 
In Section \ref{sec:oracle} we state and prove a version of Remling's Oracle Theorem (see~\cite{Remling2007,Remling2011}) suitable for the present situation. 
Section~\ref{sec:proof_main_result} contains the actual proof of the main result. 
We conclude this note with a section containing examples of couplings being admissible in the main result as well as such couplings for which the assertion of the main result fails, thus illustrating the necessity of the conditions imposed in Theorem~\ref{thm:main}.

\section{Preliminaries and main result}
\label{sec:prelim_main}

Let $\Gamma$ be a connected metric tree consisting of an infinite set of vertices $V$ and an infinite set of edges $E$, where to each $e \in E$ a length $L (e)$ is assigned; this leads to a natural metric on $\Gamma$. By $\deg (v)$ we denote the degree of a vertex $v \in V$, i.e., the number of edges which are incident to $v$. 
We assume that one distinguished vertex $O\in V$ with $\deg (O) = 1$ is denoted the root of $\Gamma$. We say that some $v \in V$ is a vertex of generation $n \in \N_0$ if the unique path connecting $v$ with $O$ contains $n$ edges.
Furthermore, for $v \in V \setminus \{O\}$ we say that $\deg (v) - 1$ is the branching number of $v$, i.e., the number of forward neighbors; moreover, we define the branching number of $O$ to be one. 
Throughout, we assume that $\Gamma$ is radial (sometimes called regular, homogeneous, or radially symmetric) (with respect to $O$), i.e., if $v, w \in V \setminus \{O\}$ are vertices of the same generation then the 
distance of $v$ to $O$ equals the distance of $w$ to $O$ and the branching number of $v$ coincides with the branching number of $w$. For $n \in \N_0$ we denote by $t_n$ the distance of an arbitrary vertex of generation $n$ to the root $O$ and by $b_n$ the branching number of an arbitrary vertex of generation $n$. In particular, $t_0 = 0$ and $b_0 = 1$. Let $v \in V$ be any vertex and let $e \in E$ be incident to $v$. We say that $e$ emanates from $v$ if $v$ is the point on $e$ with the smallest distance to $O$; otherwise we say that $e$ terminates at $v$. We identify each edge $e \in E$ with the interval $[0, L (e)]$ and assume that the endpoint zero corresponds to the vertex from which $e$ emanates. 
Furthermore, we make the following assumption.

\begin{hypothesis}\label{hyp:graph}
$\Gamma$ is an infinite, connected, regular metric tree which satisfies
\begin{equation}
\label{eq:edge}
  \tau := \inf_{n\in\N_0} (t_{n+1}-t_{n})>0.
\end{equation}
\end{hypothesis}

Note that the condition~\eqref{eq:edge} means that the set of edge lengths of $\Gamma$ is bounded away from $0$. 

We are going to study Laplacians on $\Gamma$ subject to a broad class of self-adjoint vertex conditions. To define the operators under consideration, we impose the following assumptions.

\begin{hypothesis}\label{hyp:conditions}
For each $n \in \N$,
\begin{enumerate}
 \item $\alpha_n$ and $\beta_n$ are real numbers and $\gamma_n$ is complex;
 \item $U_n$ is a unitary $(b_n - 1) \times (b_n - 1)$-matrix;
 \item $V_n$ is a $(b_n - 1) \times b_n$-matrix with orthonormal rows such that in each row the sum of the entries is zero.
\end{enumerate}
Moreover,
\begin{enumerate}
 \item[(d)] $\theta_{0, 0} \in (- \pi/2, \pi/2]$.
\end{enumerate}
\end{hypothesis}

In order to write down the vertex conditions under consideration, if $v \in V$ is a vertex of generation $n \in \N$ then we denote by $e_{v 1 +}, \dots, e_{v b_n +}$ an enumeration of the edges which emanate from $v$. Correspondingly, we write $e_{v -}$ for the unique edge which terminates at $v$.
Let $f$ be a function on $\Gamma$ such that $f|_e\in H^2(0,L(e))$ for all $e\in E$.
We denote by $f_{v j +}$ the value of $f |_{e_{v j +}}$ at $0$, $j = 1, \dots, b_n$, 
 and by $f_{v -}$ the value of $f |_{e_{v -}}$ at $L (e_{v -})$; analogously we define $f_{v 1 +}', \dots, f_{v b_n +}'$ and $f_{v -}'$. Moreover, we write $f_{v+} = (f_{v 1 +}, f_{v 2 +}, \dots, f_{v b_n +} )^\top$ and $f_{v+}' = (f_{v 1 +}', f_{v 2 +}', \dots, f_{v b_n +}')^\top$. At any vertex $v$ of generation $n \in \N$ we require vertex conditions of the form
\begin{align}\label{eq:graphCond}
\begin{split}
 \sum_{j = 1}^{b_n} f_{v j +}' - f_{v -}' & = \frac{\alpha_n}{2} \bigg( \frac{1}{b_n} \sum_{j = 1}^{b_n} f_{v j +} + f_{v -} \bigg) + \frac{\gamma_n}{2} \bigg( \sum_{j = 1}^{b_n} f_{v j +}' + f_{v -}' \bigg), \\
 \frac{1}{b_n} \sum_{j = 1}^{b_n} f_{v j +} - f_{v -} & = -\frac{\overline{\gamma_n}}{2} \bigg( \frac{1}{b_n} \sum_{j = 1}^{b_n} f_{v j +} + f_{v -} \bigg) + \frac{\beta_n}{2} \bigg( \sum_{j = 1}^{b_n} f_{v j +}' + f_{v -}' \bigg), \\
 0 & = (U_n - I) V_n f_{v+} + i (U_n + I) V_n f_{v+}'.
 \end{split}
\end{align}
Note that these conditions depend only on the generation $n\in\N$ and not on the particular vertex $v$ of generation $n$.
In addition, at the root $O$ we impose the condition
\begin{align}\label{eq:graphCondRoot}
 f' (O) + f (O) \tan \theta_{0, 0} = 0,
\end{align}
where the case $\theta_{0, 0} = \pi/2$ is to be interpreted as a Dirichlet boundary condition $f (O) = 0$. This leads to the following definition.

\begin{definition}\label{def:operator}
Let Hypothesis~\ref{hyp:graph} and~\ref{hyp:conditions} be satisfied. The operator $H_\Gamma$ in $L^2 (\Gamma)$ has the domain $\dom H_\Gamma$ which consists of all $f \in L^2 (\Gamma)$ such that
\begin{enumerate}
 \item $f |_e \in H^2 (0, L (e))$ for each $e \in E$,
 \item for each vertex $v \in V$ with generation $n \in \N$ the conditions~\eqref{eq:graphCond} are satisfied, and
 \item at the root the condition~\eqref{eq:graphCondRoot} holds.
\end{enumerate}
Moreover, the action of $H_\Gamma$ is given by
\begin{align*}
 (H_\Gamma f) |_e = - (f |_e)'' \quad \text{for each}~e \in E.
\end{align*}
\end{definition}

The operator $H_\Gamma$ is selfadjoint in $L^2 (\Gamma)$, see \cite[Lemma 2.1]{ExnerLipovsky2010}.

\begin{remark}
The conditions~\eqref{eq:graphCond} are not the most general possible for a self-adjoint vertex coupling on a metric graph. 
The vertex conditions considered here follow the radial symmetry of the tree since they depend only on the generation but not on the specific vertex. 
Moreover, by the assumption on the matrices $V_n$, the conditions~\eqref{eq:graphCond} separate those functions being radially symmetric on the subtree emanating from~$v$ from those being orthogonal to the radially symmetric functions. 
Amongst many others, the class of vertex conditions under consideration includes standard (also called Kirchhoff or natural) conditions as well as $\delta$ and weighted $\delta'$ conditions; cf.\ Section~\ref{sec:examples}.
\end{remark}

We will be interested mainly in those vertex conditions which respect the tree in the sense that they do not separate the tree into finite pieces. 
We say that the vertex conditions~\eqref{eq:graphCond} are {\em separating for generation $n\in\N$} if
\begin{align*}
 \alpha_n \beta_n + |\gamma_n|^2 = 4 \quad \text{and} \quad \Im \gamma_n = 0.
\end{align*}

The following theorem contains the main result of the present note.

\begin{theorem} \label{thm:main}
Let Hypothesis~\ref{hyp:graph} and~\ref{hyp:conditions} be satisfied and let $H_\Gamma$ be the Laplacian on $\Gamma$ in Definition~\ref{def:operator}. Assume in addition that the following assumptions hold.
\begin{enumerate}
 \item The sets $\{t_{n+1}-t_{n}:\; n\in\N\}$, $\{b_n:\; n\in\N\}$, $\{\alpha_n:\; n\in\N\}$, $\{\beta_n:\; n\in\N\}$ and $\{\gamma_n:\; n\in\N\}$ are finite. 
 Furthermore, at most finitely many $b_n$ equal~$1$.
 \item The vertex conditions~\eqref{eq:graphCond} are separating for at most finitely many generations.
 \item For all but finitely many $n \in \N$, $$ 4 (\sqrt{b_n} + 1)^2 + (\alpha_n \beta_n + |\gamma_n|^2) (\sqrt{b_n} - 1)^2 + 4 (1 - b_n) \Re \gamma_n \neq 0.$$
 \item For all but finitely many $n \in \N$, $\Re \gamma_n = 0$ and $\alpha_n \beta_n + |\gamma_n|^2 + 4 \neq 0$.
\end{enumerate}
If $H_\Gamma$ has a nonempty absolutely continuous spectrum then the sequence $\bigl((t_{n+1}-t_{n},b_n,\alpha_n,\beta_n,\gamma_n)\bigr)_n$ is eventually periodic,
i.e.~there exists $n_0\in\N$ such that from $n_0$ on this sequence is periodic.
\end{theorem}

The proof of this theorem rests on two main arguments: a unitary equivalence of $H_\Gamma$ to a direct sum of halfline operators with unique correspondence of the coupling coefficients,
and a suitable version of Remling's Oracle Theorem.

\begin{remark}
  Assumption (a) implies that the sequence of data has only finitely many values. The additional restriction on the $b_n$ excludes nontriviality of coupling for the corresponding halfline operators; cf.~\cite{BreuerFrank2009,ExnerSeifertStollmann2014}.
  Assumptions (b) and (c) are required to yield coupling conditions of one and the same form for the corresponding halfline operators, see Section \ref{sec:unitary_equivalence} below.
  Finally, assumption (d) yields injectivity of the function which maps the data sequence on the tree $\Gamma$ to the data sequence for the corresponding halfline operator, i.e.~one can then reconstruct the data on the tree by knowing the data on the halfline, see Section \ref{sec:proof_main_result}. 
  Finally, note that if the vertex conditions~\eqref{eq:graphCond} are separating for infinitely many generations then the operator $H_\Gamma$ decomposes into the direct sum of Laplacians on finite trees and thus the absolutely continuous spectrum of $H_\Gamma$ is empty, regardless of periodicity properties of the geometry or vertex conditions.
\end{remark}

\section{Unitary equivalence to halfline operators}
\label{sec:unitary_equivalence}

Our analysis of the operator $H_\Gamma$ makes use of a direct sum decomposition of $H_\Gamma$ into operators acting on half-axes; this decomposition can be found in~\cite{ExnerLipovsky2010} and it generalizes the decomposition used in~\cite{Carlson2000,NaimarkSolomyak2000,SobolevSolomyak2002,Solomyak2004} in the case of standard vertex conditions. To be more specific, assume that $\alpha_n, \beta_n, \gamma_n$ satisfy the condition~(c) of Theorem~\ref{thm:main} (without loss of generality we can assume that this condition is satisfied for all $n \in \N$) and define
\begin{align}
\label{eq:parameters}
 \mathfrak{a}_n & = \frac{16 \alpha_n}{4 (\sqrt{b_n} + 1)^2 + (\alpha_n \beta_n + |\gamma_n|^2) (\sqrt{b_n} - 1)^2 + 4 (1 - b_n) \Re \gamma_n},\nonumber \\
 \mathfrak{b}_n & = \frac{16 b_n \beta_n}{4 (\sqrt{b_n} + 1)^2 + (\alpha_n \beta_n + |\gamma_n|^2) (\sqrt{b_n} - 1)^2 + 4 (1 - b_n) \Re \gamma_n}, \\
 \mathfrak{c}_n & = 2 \frac{(1 - b_n) (4 + \alpha_n \beta_n + |\gamma_n|^2) + 8 i \sqrt{b_n} \Im \gamma_n + 4 (b_n + 1) \Re \gamma_n}{4 (\sqrt{b_n} + 1)^2 + (\alpha_n \beta_n + |\gamma_n|^2) (\sqrt{b_n} - 1)^2 + 4 (1 - b_n) \Re \gamma_n}. \nonumber
\end{align}

Furthermore, diagonalize the unitary matrices $U_n$ in the form $U_n = W_n^* D_n W_n$ with unitary $W_n$ and $D_n = \textup{diag}\, (e^{i \theta_{n, 1}}, \dots, e^{i \theta_{n, b_n - 1}})$. For $k \in \N$, at $t_k$ we consider interface conditions of the form
\begin{align}\label{eq:lineCond}
\begin{split}
 u'(t_k\rlim) - u'(t_k\llim) & = \frac{\mathfrak{a}_k}{2} \bigl(u(t_k\rlim) + u(t_k\llim)\bigr) + \frac{\mathfrak{c}_k}{2} \bigl(u'(t_k\rlim) + u'(t_k\llim)\bigr), \\
  u(t_k\rlim) - u(t_k\llim) & = -\frac{\overline{\mathfrak{c}_k}}{2} \bigl(u(t_k\rlim) + u(t_k\llim)\bigr) + \frac{\mathfrak{b}_k}{2} \bigl(u'(t_k\rlim) + u'(t_k\llim)\bigr).
\end{split}
\end{align}
The appropriate halfline operators are defined in the following way.

\begin{definition}\label{def:Hns}
Let Hypothesis~\ref{hyp:graph} and~\ref{hyp:conditions} be satisfied and let $\mathfrak{a}_n$, $\mathfrak{b}_n$, $\mathfrak{c}_n$ be defined in~\eqref{eq:parameters} for each $n \in \N$. 
Moreover, let $n \in \N$ and $s\in\set{1, \dots, b_n - 1}$, or $n=0$ and $s=0$. The domain $\dom H_{n, s}$ of the operator $H_{n, s}$ in $L^2 (t_n, \infty)$ consists of all functions $u \in L^2 (t_n, \infty)$ such that
\begin{enumerate}
 \item $u |_{(t_k, t_{k + 1})} \in H^2 (t_k, t_{k + 1})$ for each $k \geq n$,
 \item $u$ satisfies~\eqref{eq:lineCond} for each $k > n$, and
 \item $u'(t_n) +  u(t_n)\tan \theta_{n,s} = 0$.
\end{enumerate}
Moreover, the action of $H_{n, s}$ is given by
\begin{align*}
 (H_{n, s} u) |_{(t_{k}, t_{k + 1})} = - u |_{(t_{k}, t_{k + 1})}'' \quad \text{for each}~k \geq n.
\end{align*}
\end{definition}

It can be seen easily that the operators $H_{n, s}$ in $L^2 (t_n, \infty)$ are self-adjoint. The following proposition can be found in~\cite[Theorem 5.1]{ExnerLipovsky2010}. 

\begin{proposition}
\label{prop:unitary_equivalence}
Let Hypothesis~\ref{hyp:graph} and~\ref{hyp:conditions} be satisfied. Moreover, for $n \in \N$ and $s\in\set{1, \dots, b_n - 1}$, or $n=0$, $s=0$, let $H_{n, s}$ be the self-adjoint operator in Definition~\ref{def:Hns}. 
Assume that condition~(c) of Theorem~\ref{thm:main} is satisfied for each $n \in \N$. Then the space $L^2 (\Gamma)$ is unitarily equivalent to the direct sum 
\begin{align*}
 L^2 (0, \infty) \oplus \bigoplus_{n = 1}^\infty \bigoplus_{s = 1}^{b_n - 1} L^2 (t_n, \infty) \otimes \C^{b_0\cdots b_{n-1}}
\end{align*}
and, with respect to this decomposition, the operator $H_\Gamma$ is unitarily equivalent to 
\begin{align*}
 H_{0,0} \oplus \bigoplus_{n\in\N} \bigoplus_{s=1}^{b_n-1} H_{n,s} \otimes I_{\C^{b_0\cdots b_{n-1}}}.
\end{align*}
\end{proposition}

\section{A variant of Remling's Oracle Theorem}\label{sec:oracle}

In this section we prove a version of Remling's Oracle Theorem which applies to Laplacians on half-axes subject to interface conditions of the form~\eqref{eq:lineCond}. For this we rephrase these operators in the language of measures. 

\begin{definition}\label{def:mu}
Let $\tau > 0$ and let $(t_k)_{k \in \N_0}$ be a non-decreasing sequence of real numbers satisfying $\inf_{k \in \N_0} (t_{k + 1} - t_k) \geq \tau$. Moreover, for $k \in \N$ let $\mathfrak{a}_k, \mathfrak{b}_k \in \R$ and $\mathfrak{c}_k \in \C$. We define a matrix-valued Borel measure $\mu$ on the real axis by
\begin{align}\label{eq:mu}
 \mu := \sum_{n\in\N} \begin{pmatrix}
			  \mathfrak{a}_n & \mathfrak{c}_n \\ \overline{\mathfrak{c}_n} & \mathfrak{b}_n
			  \end{pmatrix} \delta_{t_n} = \begin{pmatrix}
	      \sum_{n\in\N} \mathfrak{a}_n \delta_{t_n} & \sum_{n\in\N} \mathfrak{c}_n \delta_{t_n} \\
	      \sum_{n\in\N} \overline{\mathfrak{c}_n} \delta_{t_n} & \sum_{n\in\N} \mathfrak{b}_n \delta_{t_n}
	  \end{pmatrix}.
\end{align}
The operator $H_\mu$ in $L^2 (t_0, \infty)$ formally associated with $\mu$ is defined as follows: its domain consists of all $u \in L^2 (t_0, \infty)$ such that
\begin{enumerate}
 \item $u |_{(t_k, t_{k + 1})} \in H^2 (t_k, t_{k + 1})$ for each $k \in \N_0$,
 \item $u$ satisfies~\eqref{eq:lineCond} at $t_k$ for each $k \in \N$, and
 \item $u (t_0) = 0$,
\end{enumerate}
and the action of $H_\mu$ is given by
\begin{align*}
 (H_\mu u) |_{(t_k, t_{k + 1})} = - u |_{(t_k, t_{k + 1})}'' \quad \text{for each}~k \in \N_0.
\end{align*}
\end{definition}

In analogy to the operators in Section~\ref{sec:unitary_equivalence} the operator $H_\mu$ is self-adjoint.


A mapping $\mu\from \set{B\subseteq \R:\; B\;\text{bounded Borel set}}\to \C$ is called \emph{local measure} if $\mu(\cdot\cap K)$ is a complex Radon measure for all $K\subseteq \R$ compact.
Let $\M(\R)$ be the space of local measures, and $\M(\R)^{2\times 2}$ be the 2-by-2-matrices of local measures.
For $\mu\in \M(\R)^{2\times 2}$ its variation is given by
\begin{align*}
 \abs{\mu}(A) = \sup \sum_{k=1}^n \norm{\mu(A_k)}
\end{align*}
for any Borel set $A \subset \R$, where the supremum is taken over all $n \in \N$ and all decompositions of $A$ 
into $n$ pairwise disjoint Borel sets $A_1, \dots, A_n \subset A$ with $A = \bigcup_{k = 1}^n A_k$, 
and $\|\cdot\|$ denotes the spectral norm. 
For $I\subseteq \R$ open and $C>0$ let $\M^C(I; \C^{2\times 2})$ be the set of all local matrix measures $\mu\in \M(\R)^{2\times 2}$ such that
$|\mu|$ is uniformly locally bounded by $C$, that is, $|\mu| ( (t, t + 1]) \leq C$ for all~$t \in \R$, 
and such that $\abs{\mu} (\R \setminus I) = 0$. 
It turns out that $\M^{C}(I;\C^{2\times 2})$ equipped with the topology of vague convergence is compact and hence metrizable, see \cite[Proposition 3.1]{LenzSeifertStollmann2014}.
Let $d_I$ be a metric on $\M^{C}(I;\C^{2\times 2})$ generating the topology of vague convergence.
For $I\subseteq \R$ open, $C>0$ and $\tau>0$ let $\M_{\pp}^{\tau,C} (I; \C_\textup{symm}^{2\times 2})$ be the subset of $\M^{C}(I;\C^{2\times 2})$ of measures of the form given in Definition~\ref{def:mu}. It turns out that 
$\M_{\pp}^{\tau,C} (I; \C_\textup{symm}^{2\times 2})$ is closed and hence also compact.
Note that $\M_{\pp}^{\tau,C} (I; \C_\textup{symm}^{2\times 2}) \subseteq \M_{\pp}^{\tau,C} (\R; \C_\textup{symm}^{2\times 2})$ holds for each subset $I\subseteq \R$.
Furthermore, let $\M_{\pp}^C(I; \C_\textup{symm}^{2\times 2})$ be the subset of $\M^C(I; \C^{2\times 2})$ of all pure point local matrix measures with values in $\C_\textup{symm}^{2\times 2}$.

We can now prove a version of Remling's Oracle Theorem (\cite[Theorem 2]{Remling2007}) for our setting, which generalizes \cite[Theorem 3.8]{ExnerSeifertStollmann2014} to the present, more general coupling conditions. Here we say that a measure $\mu \in \M_{\pp}^{\tau,C} \bigl(\R;\C_\textup{symm}^{2\times 2}\bigr)$ is \emph{separating for generation $n \in \N$} provided
\begin{align*}
 \mathfrak{a}_n \mathfrak{b}_n + \abs{\mathfrak{c}_n}^2 = 4 \quad \text{and} \quad \Im \mathfrak{c}_n = 0.
\end{align*}
Moreover, we write $S_x\mu:=\mu(\cdot+x)$ ($x\in\R$) for the translate of $\mu$ by $x$. By $\Sigma_{\rm ac} (H_\mu)$ we denote an essential support of the absolutely continuous part of the spectral measure of $H_\mu$. The proof of the following theorem strongly relies on observations made in~\cite{ExnerLipovsky2010}. 

\begin{theorem}
\label{thm:oracle_thm}
Let $\Lambda\subseteq \R$ be a Borel set of positive Lebesgue measure, 
and let $\varepsilon>0$, $a, b \in \R$ with $a < b$, $\tau > 0$, and $C > 0$. 
Then there exists $L_0>0$ such that for each $L \geq L_0$ there exists a continuous mapping
\begin{align*}
 \triangle\from \M_{\pp}^{\tau,C} \bigl((-L,0);\C_\textup{symm}^{2\times 2}\bigr) \to \M_{\pp}^{C} \bigl((a,b);\C_\textup{symm}^{2\times 2}\bigr)
\end{align*}
such that the following holds. If $\mu\in \M_{\pp}^{\tau,C} \bigl(\R;\C_\textup{symm}^{2\times 2}\bigr)$ is separating for at most finitely many generations and $\Sigma_{\ac} (H_\mu) \supseteq \Lambda$ then there exists $x_0 > 0$ such that
\begin{align*}
 d_{(a,b)} \bigl(\triangle(\1_{(-L,0)}S_{x}\mu), \1_{(a,b)}S_{x}\mu\bigr) < \varepsilon
\end{align*}
holds for all $x \geq x_0$.
\end{theorem}

\begin{proof}
{\bf Step~1.} This first step of the proof is of preparational nature. 
Observe first that it suffices to prove the assertion of the theorem for $\eps \leq 1$.
By compactness (and therefore equivalence of metrics), for $I\subseteq \R$ open we can use the metric $d_I$ induced by a countable subset $\set{f_{m,n}\in C_c(I):\; m,n\in\N}$,
where $\set{f_{m,n}:\; n\in\N}$ is a countable dense subset of the Banach space $C_0(K_m)$ for all $m\in\N$, and $(K_m)_{m\in\N}$ is an increasing sequence of compact subsets of $I$ whose union equals $I$.
Furthermore, for each $L > 0$ the metric $d_\R$ with this property can be chosen such that, additionally,
\begin{align}\label{eq:dominating}
 d_{(-L,0)} (\1_{(-L, 0)} \nu, \1_{(-L, 0)} \widetilde \nu) \leq d_\R (\nu, \widetilde \nu) \quad \text{and} \quad d_{(a,b)} (\1_{(a, b)} \nu, \1_{(a, b)} \widetilde \nu) \leq d_\R (\nu, \widetilde \nu)
\end{align}
holds for all $\nu, \widetilde \nu \in \M_{\pp}^{\tau,C} \bigl(\R;\C_\textup{symm}^{2\times 2})$.
We need to introduce some notation. For $\mu \in \M_{\pp}^{\tau,C} \bigl(\R;\C_\textup{symm}^{2\times 2})$, given as in Definition~\ref{def:mu}, and a fixed $t \in \R \setminus \{t_k :\; k \in \N_0\}$ we denote by $m_\pm (\cdot, t)$ the corresponding Titchmarsh--Weyl $m$-functions for the half-axes $(- \infty, t)$ and $(t, \infty)$, respectively. That is,
\begin{align*}
 m_\pm (z,t) = \pm \frac{u_{\pm}(\cdot;z)'(t)}{u_{\pm}(t;z)},
\end{align*}
where for each $z\in\C^+$ the function $u_\pm (\cdot;z)$ with $u_\pm (\cdot; z) |_{(t_k, t_{k + 1})} \in H^2 (t_k, t_{k + 1})$ for all $k \in \N$ is a solution of 
\begin{align*}
 -u'' & = zu
\end{align*}
being square integrable at $\pm\infty$, respectively, and satisfying
\begin{align*}
 u'(t_k \rlim) - u'(t_k \llim) & = \frac{\mathfrak{a}_k}{2} \bigl(u(t_k \rlim) + u(t_k \llim)\bigr) + \frac{\mathfrak{c}_k}{2} \bigl(u'(t_k \rlim) + u'(t_k \llim)\bigr),\\
 u(t_k \rlim) - u(t_k \llim) & = -\frac{\overline{\mathfrak{c}_k}}{2} \bigl(u(t_k \rlim) + u(t_k \llim)\bigr) + \frac{\mathfrak{b}_k }{2} \bigl(u'(t_k \rlim) + u'(t_k \llim)\bigr)
\end{align*}
for each $k \in \N_0$. We say that $\nu$ is {\em reflectionless} on the given Borel set $\Lambda$ if
\begin{equation}\label{eq:reflectionless}
 m_+ (E+i0,t) = -\overline{m_- (E+i0,t)} \quad \text{for a.e.}~E\in\Lambda
\end{equation}
holds for some $t \in \R \setminus \{t_k :\; k \in \N_0\}$. Note that in this case~\eqref{eq:reflectionless} holds for all~$t\in \R \setminus \{t_k :\; k \in \N_0\}$.
We denote by $\mathcal{R}^{\tau, C}(\Lambda)$ the set of all measures $\nu \in \M_{\pp}^{\tau,C} \bigl(\R;\C_\textup{symm}^{2\times 2})$ which, additionally, are reflectionless on $\Lambda$. 
Following the lines of the proof of~\cite[Proposition~2]{Remling2007} it turns out that the mapping $\1_{(-\infty,0)} \mu \mapsto \1_{(0,\infty)} \mu$, $\mu \in \mathcal{R}^{\tau, C} (\Lambda)$, is uniformly continuous w.r.t.\ $d_{(-\infty,0)}$ and $d_{(0,\infty)}$. 
By the concrete choice of these metrics at the beginning of the proof, there exists $L_0 > 0$ such that for all $L\geq L_0$ there exists and $\delta \in (0, \varepsilon/4)$ such that
\begin{align*}
 d_{(-L,0)} (\1_{(-L, 0)} \nu, \1_{(-L, 0)} \tilde{\nu}) < 5 \delta \quad \Longrightarrow \quad d_{(a,b)} (\1_{(a, b)} \nu, \1_{(a, b)} \tilde{\nu}) < \frac{\varepsilon^2}{4}
 \end{align*}
holds for all $\nu, \tilde{\nu}\in \mathcal{R}^{\tau, C} (\Lambda)$. From now on we will fix $L\geq L_0$.
  
{\bf Step~2.} In this step we define the ``oracle'' $\triangle$ and establish some of its properties. Since $\M_{\pp}^{\tau,C} \bigl((-L, 0);\C_\textup{symm}^{2\times 2}\bigr)$ is compact, also the closed $\delta$-neighborhood
\begin{align*}
 \overline U_\delta := \set{\mu \in \M_{\pp}^{\tau,C} \bigl((-L, 0);\C_\textup{symm}^{2\times 2}\bigr) :\; \exists~\nu\in \mathcal{R}^{\tau, C} (\Lambda) : d_{(-L,0)} (\mu, \1_{(-L, 0)} \nu) \leq\delta}
\end{align*}
is compact. Hence, there exists a finite set $\mathcal{F} \subseteq \mathcal{R}^{\tau, C} (\Lambda)$ such that the open balls of radius $2 \delta$ around $\mathcal{F}$ cover $\overline U_\delta$, i.e.,
\begin{align}\label{eq:covering}
 \overline U_\delta \subseteq \bigcup_{\nu\in\mathcal{F}} B(\nu,2\delta).
\end{align}
For $\nu \in \mathcal{F}$ we define $\triangle (\1_{(-L, 0)} \nu) := \1_{(a, b)} \nu$; since the measures in $\mathcal{F}$ are reflectionless, it follows in the same way as in~\cite[Proposition~2]{Remling2007} that this is well-defined. Furthermore, for $\sigma\in \overline U_\delta$ we define $\triangle(\sigma)$ as a convex combination,
\begin{align*}
 \triangle(\sigma):= \sum_{\nu\in\mathcal{F}} \frac{\bigl(3\delta-d_{(-L,0)}(\sigma,\1_{(-L,0)}\nu)\bigr)^+}{\sum_{\nu\in\mathcal{F}} \bigl(3\delta-d_{(-L,0)}(\sigma,\1_{(-L,0)}\nu)\bigr)^+} \triangle(\1_{(-L, 0)} \nu),
\end{align*}
where $(\cdot)^+$ denotes the positive part. Then clearly $\triangle(\sigma)\in \M_{\pp}^{C} \bigl((a, b);\C_\textup{symm}^{2\times 2}\bigr)$ and $\triangle\from \overline U_\delta \to \M_{\pp}^{C} \bigl((a, b);\C_\textup{symm}^{2\times 2}\bigr)$ is continuous. Moreover, we observe that for $\tilde \nu \in \mathcal{F}$ we have 
\begin{align}\label{eq:folgerung}
 d_{(-L,0)} (\sigma, \1_{(-L, 0)} \tilde{\nu}) < 2 \delta \quad \Longrightarrow \quad d_{(a,b)} (\triangle(\sigma), \1_{(a, b)} \tilde{\nu}) < \frac{\varepsilon}{2}.
\end{align}
Indeed, given $\tilde{\nu}\in\mathcal{F}$ with $d_{(-L,0)} (\sigma, \1_{(-L, 0)} \tilde{\nu}) < 2 \delta$, by the triangle inequality we have $d_{(-L,0)} (\1_{(-L, 0)} \nu, \1_{(-L, 0)} \tilde{\nu}) < 5 \delta$ for all $\nu \in \mathcal{F}$ contributing to the sum. Thus, Step~1 implies $d_{(a,b)} (\1_{(a, b)} \nu,\1_{(a, b)} \tilde{\nu}) < \varepsilon^2 / 4$ for these $\nu$, and by a reasoning as in~\cite[Lemma 2]{Remling2007} we obtain
\begin{align*}
 d_{(a,b)} (\triangle(\sigma), \1_{(a, b)} \tilde{\nu}) < \frac{3}{8} \eps^4 \ln (16 \eps^{-4}) < \frac{\varepsilon}{2}.
\end{align*}
(Note that for all $\sigma\in \overline U_\delta$ there exists $\tilde{\nu}\in\mathcal{F}$ with $d_{(-L,0)}(\sigma,\tilde{\nu}_-)\leq 2\delta$.) 
Now the extension theorem of Dugundji and Borsuk \cite[Chapter II, Theorem 3.1]{BessagaPelczynski1975} yields a continuous extension of $\triangle$ to $\M_{\pp}^{\tau,C} \bigl((-L, 0);\C_\textup{symm}^{2\times 2}\bigr)$ via convex combinations, and thus this extension again maps into $\M_{\pp}^{C} \bigl((a, b);\C_\textup{symm}^{2\times 2}\bigr)$.
  
{\bf Step~3.} It remains to show that the mapping $\triangle$ has the desired properties. 
Let $\mu\in \M_{\pp}^{\tau,C} \bigl(\R;\C_\textup{symm}^{2\times 2}\bigr)$ with $\Sigma_{\ac}(H_\mu)\supseteq \Lambda$. 
Due to the fact that $\mu$ is separating for at most finitely many generations, by~\cite[Theorem 6.2]{ExnerLipovsky2010} the so-called $\omega$-limits of $\mu$ are reflectionless on $\Lambda$, that is, there exists $x_0 > 0$ such that for each $x \geq x_0$ we have some $\nu \in \mathcal{R}^{\tau, C} (\Lambda)$ with $d_\R(S_x\mu,\nu) < \delta$. Then, by~\eqref{eq:dominating}, 
\begin{align}\label{eq:desBrauchma}
 d_{(-L,0)} (\1_{(-L, 0)} S_x \mu, \1_{(-L, 0)} \nu) < \delta \quad \text{and} \quad d_{(a,b)} (\1_{(a, b)} S_x \mu, \1_{(a, b)} \nu) < \delta
\end{align}
holds for this $\nu$. Hence, $\1_{(-L, 0)} S_x \mu \in \overline U_\delta$ and due to~\eqref{eq:covering} there exists $\tilde{\nu}\in \mathcal{F}$ such that
\begin{align}\label{eq:naSowas}
 d_{(-L,0)} (\1_{(-L, 0)} S_x \mu, \1_{(-L, 0)} \tilde{\nu}) < 2\delta.
\end{align}
By~\eqref{eq:folgerung} we thus obtain 
\begin{align}\label{eq:desA}
 d_{(a,b)} (\triangle(\1_{(-L, 0)} S_x\mu), \1_{(a, b)} \tilde{\nu} ) < \varepsilon / 2.
\end{align}
Note that by~\eqref{eq:naSowas} and~\eqref{eq:desBrauchma} we have 
\begin{align*}
 & d_{(-L,0)} (\1_{(-L, 0)} \nu, \1_{(-L, 0)} \tilde \nu)\\
 & \leq d_{(-L,0)} (\1_{(-L, 0)} \nu, \1_{(-L, 0)} S_x \mu) + d_{(-L,0)} (\1_{(-L, 0)} S_x \mu, \1_{(-L, 0)} \tilde \nu) \\
 & < 3 \delta.
\end{align*}
Therefore Step~1 implies $d_{(a,b)} (\1_{(a, b)} \nu, \1_{(a, b)} \tilde{\nu}) < \varepsilon^2 / 4$ and we can conclude with the help of~\eqref{eq:desA} and~\eqref{eq:desBrauchma}
\begin{align*}
 d_{(a,b)} \bigl(\triangle(\1_{(-L, 0)} S_x\mu), \1_{(a, b)} (S_x\mu)\bigr) < \frac{\varepsilon}{2} + \frac{\varepsilon^2}{4} + \delta < \varepsilon.
\end{align*}
This completes the proof.
\end{proof}

\section{Proof of the main result}
\label{sec:proof_main_result}

In this section we prove the main result of this note. The proof is carried out in four steps.

{\bf Step~1.} As $\sigma_{\rm ac} (H_\Gamma) \neq \varnothing$, it follows from Proposition~\ref{prop:unitary_equivalence} that there exist $n \in \N_0$ and $s \in \set{0,\ldots,b_n-1}$ such that the operator $H_{n, s}$ in Definition~\ref{def:Hns} has nonempty absolutely continuous spectrum. Thus, if we replace the boundary condition of the functions in the domain of $H_{n, s}$ at $t_n$ by a Dirichlet boundary condition (which is a rank one perturbation in the resolvent sense and, hence, does not change the absolutely continuous spectrum), it follows that the operator $H_\mu$ in Definition~\ref{def:mu} associated with the sequence $(t_n, t_{n + 1}, \dots)$ with $\mu$ given in~\eqref{eq:mu} and $\mathfrak{a}_k, \mathfrak{b}_k, \mathfrak{c}_k$ given in~\eqref{eq:parameters}, $k \geq n$, has a nonempty absolutely continuous spectrum. 

{\bf Step~2.} In this step we show that $\mu$ is separating for at most finitely many generations. For this let first $k \in \N$ be arbitary and let us calculate how the coefficients $\alpha_k, \beta_k, \gamma_k$ and the branching numbers $b_k$ can be recovered from $\mathfrak{a}_k, \mathfrak{b}_k, \mathfrak{c}_k$. We distinguish two cases. First, if $\Re \mathfrak{c}_k = 0$ then the last identity in~\eqref{eq:parameters} together with the first assumption in condition~(d) of the theorem yields
\begin{align*}
 (1 - b_k) (\alpha_k \beta_k + |\gamma_k|^2 + 4) = 0.
\end{align*}
From this and the second assumption in~(d) it follows $b_k = 1$. Hence~\eqref{eq:parameters} yields $\alpha_k = \mathfrak{a}_k$, $\beta_k = \mathfrak{b}_k$ and $\gamma_k = \mathfrak{c}_k$. In the second case, $\Re \mathfrak{c}_k \neq 0$, the last identity in~\eqref{eq:parameters} implies $b_k > 1$ and
\begin{align}\label{eq:weiterSo}
 \alpha_k \beta_k + (\Im \gamma_k)^2  = \frac{8 (1 - b_k) - 4 \Re \mathfrak{c}_k (\sqrt{b_k} + 1)^2}{\Re \mathfrak{c}_k (\sqrt{b_k} - 1)^2 - 2 (1 - b_k)},
\end{align}
where the denominator is nonzero since otherwise the last identity in~\eqref{eq:parameters} would imply $\sqrt{b_k} \Re \mathfrak{c}_k = 0$, a contradiction. Plugging~\eqref{eq:weiterSo} in the identities~\eqref{eq:parameters} we get
\begin{align}\label{eq:schlimmSchlimm}
\begin{split}
 \alpha_k = \frac{\mathfrak{a}_k}{16} \mathcal{I}_k, \quad \beta_k = \frac{\mathfrak{b}_k}{16 b_k} \mathcal{I}_k, \quad \text{and} \quad \Im \gamma_k = \frac{\Im \mathfrak{c}_k}{16 \sqrt{b_k}} \mathcal{I}_k,
\end{split}
\end{align}
where 
\begin{align*}
 \mathcal{I}_k = - \frac{32 (1 - b_k) \sqrt{b_k}}{\Re \mathfrak{c}_k (\sqrt{b_k} - 1)^2 - 2 (1 - b_k)}.
\end{align*}
In particular,
\begin{align*}
 \alpha_k \beta_k + (\Im \gamma_k)^2 = & \frac{\mathfrak{a}_k \mathfrak{b}_k + (\Im \mathfrak{c}_k)^2}{256 b_k} \mathcal{I}_k^2.
\end{align*}
Comparing this identity to~\eqref{eq:weiterSo} we find
\begin{align}\label{eq:auweia}
 - 1024 b_k (b_k - 1) \left( (b_k - 1) \left( \mathfrak{a}_k \mathfrak{b}_k + |\mathfrak{c}_k|^2 + 4 \right) + 4 \Re \mathfrak{c}_k (b_k + 1) \right) = 0,
\end{align}
which has the unique solution
\begin{align}\label{eq:solution}
 b_k = \frac{|2 - \mathfrak{c}_k|^2 + \mathfrak{a}_k \mathfrak{b}_k}{|2 + \mathfrak{c}_k|^2 + \mathfrak{a}_k \mathfrak{b}_k},
\end{align}
where the denominator is nonzero since otherwise~\eqref{eq:auweia} would imply $\Re \mathfrak{c}_k = 0$.

Assume now that the measure $\mu$ is separating for some generation $k \geq n$. Then $\Im \mathfrak{c}_k = 0$ and it follows from the last equation in~\eqref{eq:parameters} that $\Im \gamma_k = 0$. Let us again distinguish two cases. If $\Re \mathfrak{c}_k = 0$ then the above considerations yield
\begin{align*}
 \alpha_k \beta_k + |\gamma_k|^2 = \mathfrak{a}_k \mathfrak{b}_k + |\mathfrak{c}_k|^2 = 4.
\end{align*}
In the other case, $\Re \mathfrak{c}_k \neq 0$, from~\eqref{eq:schlimmSchlimm} we obtain
\begin{align}\label{eq:soIsses}
 \alpha_k \beta_k + |\gamma_k|^2 & = \frac{\mathfrak{a}_k \mathfrak{b}_k}{256 b_k} \mathcal{I}_k^2 = \frac{4 - \mathfrak{c}_k^2}{256 b_k} \mathcal{I}_k^2 = - \frac{4 (\sqrt{b_k} + 1)^2 (\mathfrak{c}_k^2 - 4)}{(2 - \mathfrak{c}_k + \sqrt{b_k} (2 + \mathfrak{c}_k) )^2}.
\end{align}
As~\eqref{eq:solution} and $\mathfrak{a}_k \mathfrak{b}_k = 4 - \mathfrak{c}_k^2$ imply $b_k = \frac{2 - \mathfrak{c}_k}{2 + \mathfrak{c}_k}$, using~\eqref{eq:soIsses} we arrive again at $\alpha_k \beta_k + |\gamma_k|^2 = 4$. Thus the conditions~\eqref{eq:graphCond} are separating for generation $k$ in both cases. Since by assumption this is the case for at most finitely many $k \in \N$, it follows that $\mu$ is separating for at most finitely many generations.

{\bf Step~3.} In this step we show that the measure $\mu$ in Definition~\ref{def:mu} is eventually periodic, i.e., there exist $x, y \in \R$ with $x \neq y$ such that $\1_{[x, \infty)} \mu$ is a translate of $\1_{[y, \infty)} \mu$. Let $T:=\set{t_k:\; k\in\N_0}$ and $\ell := \max\{ t_{k+1}-t_k:\; n\in\N_0 \} + 1$. Moreover, let
\begin{align*}
 \mathcal{D}:= \set{\1_{(-1,\ell)}(S_{t}\mu):\; t\in T}.
\end{align*}
It follows from the assumption~(a) that the set $\mathcal{D}$ is finite and hence there exists $\varepsilon > 0$ such that 
\begin{align}\label{eq:DFinite}
 \nu_1, \nu_2 \in \mathcal{D}, \nu_1 \neq \nu_2 \quad \Longrightarrow \quad d (\nu_1, \nu_2) > 2 \varepsilon.
\end{align}
By Step~1 and~2, all assumptions of Theorem~\ref{thm:oracle_thm} are satisfied for the measure $\mu$ and the operator $H_\mu$ and hence there exist $L > \ell$, $x_0 > 0$, and a continuous function $\triangle\from \M_{\pp}^{\tau,C} \bigl((-L,0);\C_\textup{symm}^{2\times 2}\bigr) \to \M_{\pp}^{C} \bigl((-1, \ell);\C_\textup{symm}^{2\times 2}\bigr)$ such that
\begin{align}\label{eq:oracle}
 d_{(-1,\ell)} (\triangle (\1_{(- L, 0)} S_x \mu), \1_{(-1, \ell)} S_x \mu) \leq \eps \quad \text{for all}~x \geq x_0.
\end{align}
Due to the assumption~(a) of the theorem, there exist indices $k, K, m, M \in \N_0$ with $k \neq m$, $K > k$ and $M > m$ such that $t_K - t_k \geq L$, $t_M - t_m \geq L$, and
\begin{align}\label{eq:KProperty}
 S_{t_k}(\1_{[t_k, t_K)} \mu) = S_{t_m}(\1_{[t_m, t_M)}\mu);
\end{align}
here the indices can be chosen such that $y^1 := t_K > x_0$ and $z^1 := t_M > x_0$. By~\eqref{eq:KProperty} it follows $\1_{(- L, 0)} S_{y^1} \mu = \1_{(- L, 0)} S_{z^1} \mu$. Hence~\eqref{eq:oracle} yields
\begin{align*}
 d_{(-1,\ell)}\bigl(\1_{(-1,\ell)} S_{y^1} \mu, \1_{(-1,\ell)} S_{z^1}\mu \bigr) \leq 2 \varepsilon.
\end{align*}
Combining this fact with~\eqref{eq:DFinite} we obtain 
\begin{align}\label{eq:extended}
 \1_{(-1,\ell)} S_{y^1}\mu = \1_{(-1,\ell)} S_{z^1}\mu.
\end{align}
Define $y^2:=\min (T \cap (y^1,\infty))$ and $z^2:=\min (T \cap (z^1,\infty))$. Then~\eqref{eq:extended} and the definition of $\ell$ imply
\begin{align*}
 z^2-z^1 = y^2-y^1 \quad \text{and} \quad S_{y^1}(\1_{[y^1,y^2)}\mu) = S_{z^1}(\1_{[z^1,z^2)}\mu).
\end{align*}
Together with~\eqref{eq:KProperty} we even have
\begin{align*}
 S_{t_k} \1_{[t_k, y^2)} \mu = S_{t_m} \1_{[t_m, y^2)};
\end{align*}
in particular, $\1_{(- L, 0)} S_{y^2} \mu = \1_{(-L, 0)} S_{z^2} \mu$ holds and we can apply~\eqref{eq:oracle} and~\eqref{eq:DFinite} to obtain $\1_{(-1, \ell)} S_{y^2} \mu = \1_{(-1, \ell)} S_{z^2} \mu$. Thus, with $y^3 := \min (T \cap (y^2, \infty))$ and $z^3 := \min (T \cap (z^2, \infty))$ we observe that $y^3 - y^2 = z^3 - z^2$ and that $\1_{[t_k, y^3)} \mu$ and $\1_{[t_k, z^3)} \mu$ are translates of each other. Iterating this procedure, we obtain sequences $(y^n)$ and $(z^n)$ in $T$ tending to $+ \infty$ such that $\1_{[t_k, y^n)} \mu$ is a translate of $\1_{[t_m, z^n)}\mu$ and $y^{n+1} - y^n = z^{n+1} - z^n$ for all $n\in\N$. Hence $\1_{[t_k,\infty)}\mu$ is a translate of $\1_{[t_m,\infty)}\mu$, i.e., $\mu$ is eventually periodic.

{\bf Step~4.} 
First, note that $b_k>1$ implies $\mathfrak{c}_k\neq 0$ by the first part of Step~2.
In order to obtain the claim of the theorem, note that, therefore, by Step~3 the sequence $\bigl( (t_{k + 1} - t_k, \mathfrak{a}_k, \mathfrak{b}_k, \mathfrak{c}_k) \bigr)_{k \in \N}$ is eventually periodic. Moreover, by Step~2 the coefficients $\alpha_k, \beta_k, \gamma_k$ and the branching numbers $b_k$ are uniquely determined by this data, and thus also the sequence $\bigl( (t_{k + 1} - t_k, b_k, \alpha_k, \beta_k, \gamma_k) \bigr)_{k \in \N}$ is eventually periodic. This completes the proof of Theorem~\ref{thm:main}.

\section{Examples}
\label{sec:examples}

In this section we provide examples of classes of vertex conditions to which Theorem~\ref{thm:main} can be applied. Moreover, we provide counterexamples to the result in cases where several of our assumptions are violated.

\begin{example}
Let $V_n$ be any $(b_n - 1) \times b_n$-matrix according to the assumptions of Hypothesis~\ref{hyp:conditions}~(c). Then the choice $U_n = - I \in \C^{(b_n - 1) \times (b_n - 1)}$, $\alpha_n = \beta_n = \gamma_n = 0$ corresponds to standard (also called Kirchhoff or natural) vertex  conditions
\begin{align*}
 f_{v-} = f_{v1+} = \dots = f_{vb_n+} \quad \text{and} \quad \sum_{j=1}^{b_n} f_{vj+}'-f_{v-}' = 0
\end{align*}
at all vertices $v$ of generation $n$. The same choice of $U_n$ together with $\beta_n = \gamma_n = 0$ and arbitrary $\alpha_n \in \R$ leads to so-called $\delta$-couplings,
\begin{align*}
 f_{v-} = f_{v1+} = \dots = f_{vb_n+} \quad \text{and} \quad \sum_{j=1}^{b_n} f_{vj+}'-f_{v-}' = \alpha_n f_{v-}.
\end{align*}
Moreover, with $U_n = I$, $\alpha_n = \gamma_n = 0$ and $\beta_n \in \R$ arbitrary we arrive at weighted $\delta'$ conditions
\begin{align*}
 f_{v1+}' = \dots = f_{vb_n+}' = \frac{1}{b_n} f_{v-}' \quad \text{and} \quad \frac{1}{b_n} \sum_{j=1}^{b_n} f_{vj+} - f_{v-} = \beta_n f_{v-}'.
\end{align*}
If for each $n \in \N$ one of these types of conditions is chosen then the assumptions (b)--(d) of Theorem \ref{thm:main} are satisfied. 
Note that in the special case of standard vertex conditions at all vertices Theorem~\ref{thm:main} reduces to~\cite[Theorem 5.1]{ExnerSeifertStollmann2014}. 
We remark that in the contex of quantum graphs the above types of coupling conditions are frequently used; cf.~\cite{Kuchment2007}.
\end{example}

The following two examples show that the condition~(d) in Theorem~\ref{thm:main} is needed in order to conclude eventual periodicity of $(b_n)$ and the coefficient sequences $(\alpha_n), (\beta_n)$ and $(\gamma_n)$. Indeed, the first example shows that in the case $\Re \gamma_n \neq 0$ there can be absolutely continuous spectrum for the operator on a tree with an aperiodic sequence of coefficients taking only finitely many values.

\begin{example}
Let $\alpha_n = \beta_n = 0$ for all $n \in \N$ and 
\begin{align*}
  (\gamma_n, b_n) = 
 \begin{cases}
  (2/3, 4), & \text{if}~n = 2^k~\text{for some}~k \in \N, \\
  (1, 9), & \text{else}.
 \end{cases}
\end{align*}
Then  the conditions~(a),~(b) and~(c) of Theorem~\ref{thm:main} are satisfied. Moreover, direct computation according to~\eqref{eq:parameters} yields $\mathfrak{a}_n = \mathfrak{b}_n = \mathfrak{c}_n = 0$ for all $n \in \N$. 
Thus, for $\Re \gamma_n \neq 0$ periodicity of $( (\mathfrak{a}_n, \mathfrak{b}_n, \mathfrak{c}_n))_n$ does not necessarily imply periodicity of $(\alpha_n)_n, (\beta_n)_n$ and $(\gamma_n)_n$. 
Furthermore, the Hamiltonian corresponding to these parameters and the Hamiltonian corresponding to $\widetilde \alpha_n = \widetilde \beta_n = 0$, $\widetilde \gamma_n = 2/3$, $\widetilde b_n = 4$ for all $n \in \N$ are unitarily equivalent by Proposition~\ref{prop:unitary_equivalence} and, hence, have the same absolutely continuous spectrum. 
Since the halfline operator corresponding to $\mathfrak{a}_n = \mathfrak{b}_n = \mathfrak{c}_n = 0$ for all $n \in \N$  is given by the free Laplacian on $(0,\infty)$ (with a boundary condition at zero determined by $\theta$), the absolutely continuous spectrum equals $[0, \infty)$ and, in particular, is nonempty. 
We remark that this reasoning is independent of the periodicity of the sequence $(t_{n + 1} - t_n)$; cf.\ also~\cite[Example~7.1]{ExnerLipovsky2010}.
\end{example}

The second example shows that in the case $\alpha_n\beta_n+\abs{\gamma_n}^2 + 4 = 0$ a similar phenomenon may occur.

\begin{example}
Set $\alpha_n = 4, \beta_n = -1, \gamma_n = 0, b_n = 4$ if $n = 2^k$ for some $k \in \N$ and $\alpha_n = 6, \beta_n = -2/3, \gamma_n = 0, b_n = 9$ otherwise. 
Then the conditions~(b)--(c) of Theorem~\ref{thm:main} are fulfilled and $\alpha_n \beta_n + |\gamma_n|^2 + 4 = 0$ for all $n \in \N$.
Moreover, $\mathfrak{a}_n = 2$, $\mathfrak{b}_n = -2$ and $\mathfrak{c}_n = 0$ for all $n \in \N$, and the same result appears in the case $\widetilde \alpha_n = 6, \widetilde \beta_n = -2/3, \widetilde \gamma_n = 0, \widetilde b_n = 9$ for all $n \in \N$. 
Thus by Proposition~\ref{prop:unitary_equivalence} we have two Hamiltonians being unitarily equivalent, one of them having periodic coefficients and branching numbers, the other one not. 
Hence their absolutely continuous spectra coincide. 
Note that in this case the coupling conditions for the halfline operators in Definition~\ref{def:Hns} and Proposition~\ref{prop:unitary_equivalence} are given by
\begin{align*}
  u(t_n+) & = -u'(t_n-),\\
  u'(t_n+) & = u(t_n-).
\end{align*}
\end{example}

\end{document}